\documentclass[reqno,a4paper,11pt]{amsart}

\usepackage{amsmath,amstext,amsfonts,amsbsy,eucal,amssymb}
\usepackage[latin1]{inputenc}

\numberwithin{equation}{section}

\newtheorem{theorem}{Theorem}[section]
\newtheorem{lemma}[theorem]{Lemma}

\newtheorem{proposition}[theorem]{Proposition}

\theoremstyle{definition}
\newtheorem{definition}[theorem]{Definition}

\begin{document}

\parskip 4pt
\baselineskip 16pt


\title[On a certain arithmetic function defined via Bernoulli numbers]
{On a certain arithmetic function defined via Bernoulli numbers}

\author[Andrei K. Svinin]{Andrei K. Svinin}
\address{Matrosov Institute for System Dynamics and Control Theory of Siberian Branch of Russian Academy of Sciences, P.O. Box 292, 664033 Irkutsk, Russia}
\email{svinin@icc.ru}

\date{\today}



\begin{abstract}
The paper defines a certain function \(\chi(n)\) on the set of odd integers \(n \ge 3\). 
The integrality property of this function is investigated. 
The main result establishes that the function is integer-valued if and only if 
\(n\) is a prime number or an odd Giuga number. 
Thus, the integrality of \(\chi(n)\) serves as a single criterion unifying three classes: 
prime numbers, Carmichael numbers, and odd Giuga numbers. 
The proof relies exclusively on the von Staudt--Clausen theorem on the denominators 
of Bernoulli numbers and Korselt's criterion for Carmichael numbers, which makes 
the exposition elementary.

A connection is proven between prime numbers \(p\) for which \(\chi(p)=p-1\) and 
Nov\'ak--Carmichael numbers. 
It is proven that the primes satisfying this criterion form a sequence with the 
identifier A337119 in the Online Encyclopedia of Integer Sequences.

For the classes \(P_\mu\) of primes with a fixed value \(\chi(p)=\mu\), 
a characterization is obtained in terms of the exponents of the prime divisors of 
the numerator of the corresponding Bernoulli number \(B_{p-1}\). 
For some \(\mu\), explicit congruences are provided.

The work is accessible to a wide readership due to its minimal use of advanced 
mathematical apparatus.

\end{abstract}

\maketitle

\section*{Introduction}

As is well known, the Bernoulli numbers \(B_j\) are defined by the exponential generating series
\[
\frac{t}{e^t-1}=\sum_{j\geq 0} B_j\frac{t^j}{j!}.
\]
Bernoulli numbers play a fundamental role in number theory. They are closely connected with the Riemann zeta function: for positive even arguments the values \(\zeta(2k)\) are expressed via \(B_{2k}\). Classical results concerning these numbers include the von Staudt-Clausen theorem, which completely describes the denominators of Bernoulli numbers; the Kummer criterion for regular primes, which is directly related to Fermat's Last Theorem. Furthermore, Bernoulli numbers appear in the theory of congruences, in the Euler-Maclaurin summation formula, and in the computation of class numbers of quadratic fields. Their arithmetic properties, namely the behaviour of numerators and denominators, remain the subject of active research.
\begin{theorem}[von Staudt-Clausen]
For any positive odd integer $n$ the number
\[
B_{n-1}+\sum_{p} \frac{1}{p},
\]
where the sum runs over all primes $p$ satisfying $(p-1)\mid (n-1)$, is an integer.
\end{theorem}

Carmichael numbers and odd Giuga numbers are examples of pseudoprimes. It turns out that they can be characterized by a single arithmetic condition if one uses a special function constructed with the help of Bernoulli numbers. In the present work we propose such a function and show that its integrality serves, in a sense, as a unified criterion for these classes of odd numbers.
\begin{definition}
A composite positive integer $n$ is called a Carmichael number if
\begin{equation}
a^{n-1}\equiv 1\pmod n
\label{88876}
\end{equation}
for all integers $a$ coprime to $n$.
\end{definition}
It suffices to check this congruence for all reduced residue classes modulo $n$, i.e. for \(1\leq a\leq n-1,\; \gcd(a,n)=1\).
Such numbers are also called absolute pseudoprimes, because they pass the Fermat test for every base. It is known that all Carmichael numbers are odd and have at least three prime factors. The smallest such number is \(561=3\cdot 11\cdot 17\). The question of their infinitude was settled positively only in 1994 \cite{Alford}.

Direct verification of the condition in the definition of Carmichael numbers is cumbersome, so one may use Korselt's criterion \cite{Crandall, Korselt}.
\begin{theorem}[Korselt's criterion]
A composite number $n$ is a Carmichael number if and only if it is squarefree and for every prime divisor $p$ of $n$ we have \((p-1)\mid (n-1)\).
\end{theorem}
While Carmichael numbers are characterized by congruence (\ref{88876}), Giuga numbers are defined by a finer condition.
\begin{definition}
A Giuga number is a composite integer $n$ such that every prime divisor $p$ of $n$ satisfies
\begin{equation}
p\mid \left(\frac{n}{p}-1\right) \;\Leftrightarrow\; p^2\mid (n-p) \;\Leftrightarrow\; n\equiv p \pmod{p^2}.
\label{867434}
\end{equation}
\end{definition}
All known Giuga numbers are even, and the question of the existence of odd Giuga numbers is open and closely related to the Agoh-Giuga conjecture. It can be shown that any odd Giuga number (if it exists) is automatically a Carmichael number. Lower bounds for the size of such numbers are known (see \cite{Borwein,Borwein1,Kellner1}).

Now we define the function that will, in a certain sense, be related to these classes of numbers. For every odd integer $n\ge 3$ set\footnote{Here and below \(\mathrm{num}(r)\)/\(\mathrm{denom}(r)\) denote the numerator/denominator of a rational number $r$ in lowest terms.}
\[
\chi(n)=\frac{\mathrm{denom}\left(\displaystyle\frac{B_{n-1}}{n-1}\right)}{n\cdot\mathrm{denom}\left(B_{n-1}+\displaystyle\frac{1}{n}\right)}.
\]
We shall also consider the function \(\tilde{\chi}(n)=n\cdot \chi(n)\).

Our main goal is to study integrality properties of the function \(\chi(n)\). The main result of the work is a theorem stating that \(\chi(n)\) is integer-valued if and only if $n$ is a prime number or an odd Giuga number. The paper is written at an elementary level and relies only on the von Staudt-Clausen theorem and Korselt's criterion, both stated above, which makes it accessible to a wide readership.

\section{Integrality properties of the function $\chi(n)$}
\label{67540}

In this section we prove lemmas that together give a complete characterization of the integrality of \(\chi(n)\): for primes, for composite non-Carmichael numbers, and for Carmichael numbers. Viewing integrality as a test, we show that the only composite numbers that can "fool" this test are the hypothetical odd Giuga numbers.

\begin{lemma}
For every prime $p\geq 3$ the value $\chi(p)$ is an integer.
\end{lemma}
\begin{proof}
From the complementarity property for denominators of Bernoulli numbers (a consequence of the von Staudt-Clausen theorem) we have
\[
\mathrm{denom}\left(B_{p-1}+\frac{1}{p}\right)=\frac{\mathrm{denom}(B_{p-1})}{p}.
\]
Also we use the general identity
\begin{equation}
\mathrm{denom}\left(\frac{B_{n-1}}{n-1}\right)=\frac{(n-1)\mathrm{denom}(B_{n-1})}{\gcd(n-1, \mathrm{num}(B_{n-1}))}
\label{90865}
\end{equation}
for $n=p$. Clearly,
\begin{equation}
\chi(p)=\frac{p-1}{\gcd(p-1, \mathrm{num}(B_{p-1}))}
\label{675498}
\end{equation}
is an integer.
\end{proof}

\begin{lemma} \label{875643}
For every prime $p$ the value $\chi(p)$ is even.
\end{lemma}
\begin{proof}
By the von Staudt-Clausen theorem the denominator of $B_{p-1}$ is always divisible by $2$, hence the numerator is odd. Since $p-1$ is even, from formula (\ref{675498}) it follows that $\chi(p)$ is even.
\end{proof}

Recall that \(\tilde{\chi}(n)=n\cdot \chi(n)\). Using (\ref{90865}) and the relation\footnote{For further computations it is convenient to set
\(
A=\mathrm{num}(B_{n-1})
\)
and
\(
D=\mathrm{denom}(B_{n-1}).
\)
}
\[
\mathrm{denom}\left(B_{n-1}+\frac{1}{n}\right)=\frac{D\cdot n}{\gcd(An+D, D\cdot n)},
\]
we obtain
\begin{equation}
\tilde{\chi}(n)=\frac{(n-1)\gcd(An+D, D\cdot n)}{n\gcd(A, n-1)}.
\label{6754900008}
\end{equation}

\begin{lemma}
If $n$ is composite and not a Carmichael number, then $\chi(n)\notin\mathbb Z$.
\end{lemma}
\begin{proof}
Let \(q\) be a prime divisor of \(n\) such that \((q-1)\nmid(n-1)\). Note that the existence of such a \(q\) is equivalent to \(n\) not being a Carmichael number. By the von Staudt-Clausen theorem $q\nmid D$. Since $q\mid n$, we have $q\nmid(n-1)$, hence $q\nmid \gcd(A,n-1)$. Moreover, since $q\nmid D$ and $q\mid n$,
\(
An+D \equiv D \not\equiv 0 \pmod q,
\)
so $v_q(\gcd(An+D, D\cdot n))=0$.\footnote{Here and below \(v_q(n)\) denotes the exponent of the prime \(q\) in the prime factorization of \(n\).} Then the exponent of $q$ in the denominator of (\ref{6754900008}) equals
\[
v_q(n)+v_q(\gcd(A,n-1))-v_q(\gcd(An+D, D\cdot n)) = v_q(n) > 0.
\]
Consequently, $\tilde{\chi}(n)\notin\mathbb Z$, hence $\chi(n)=\tilde{\chi}(n)/n$ is not an integer.
\end{proof}

Now we consider Carmichael numbers. For them, as will be shown, \(\tilde{\chi}(n)\) is always integral, although \(\chi(n)\) itself is not integral unless $n$ is an odd Giuga number. We formulate this in the next two lemmas.

\begin{lemma}
For every Carmichael number $n$ the value $\tilde{\chi}(n)$ is an integer.
\end{lemma}
\begin{proof}
Notice that $\gcd(A, n-1)$ is coprime to $n$. Indeed, if $p\mid n$, then by the von Staudt-Clausen theorem $p\mid D$, hence $p\nmid A$. Also $\gcd(n,n-1)=1$, so $p\nmid(n-1)$. Therefore $p\nmid \gcd(A,n-1)$, and so $\gcd(A,n-1)$ is coprime to $n$. Thus from (\ref{6754900008}) it follows that $\tilde{\chi}(n)\in\mathbb Z$ iff $n\mid \gcd(An+D, D\cdot n)$, which is equivalent to $n\mid (An+D)$, and this is equivalent to $n\mid D$. For Carmichael numbers $n\mid D$ (this follows from the von Staudt-Clausen theorem). Hence $\tilde{\chi}(n)$ is an integer.
\end{proof}

Now we establish the connection with Giuga numbers. For a Carmichael number $n$ write $D=n\cdot m$, where $\gcd(n,m)=1$. Then
\[
B_{n-1}+\frac{1}{n} = \frac{A+m}{nm},
\]
and from (\ref{6754900008}) we obtain
\begin{equation}
\tilde{\chi}(n)=\frac{(n-1)\gcd(A+m, n)}{\gcd(A, n-1)}.
\label{6753867}
\end{equation}

\begin{lemma}
\label{7834567}
For a Carmichael number $n$, the condition $n\mid \tilde{\chi}(n)$ is equivalent to $n$ being a Giuga number.
\end{lemma}
\begin{proof}
It is enough to show that $n\mid (A+m)$ is equivalent to the Giuga condition (\ref{867434}). By the von-Staudt-Clausen theorem, for a Carmichael number we have the congruence
\[
p\cdot B_{n-1}\equiv -1 \pmod p,\; \forall p\mid n.
\]
Since $B_{n-1}=A/(nm)$, we get
\[
\frac{pA}{nm}\equiv -1 \pmod p \;\Leftrightarrow\; A \equiv -\frac{n}{p}m \pmod p.
\]
Then
\[
A+m \equiv m\left(1-\frac{n}{p}\right) \pmod p.
\]
Because $p\nmid m$ (as $\gcd(n,m)=1$), we have $p\mid (A+m)$ iff $n/p\equiv 1\pmod p$, i.e. $n\equiv p\pmod{p^2}$. Since $n$ is squarefree, this must hold for all $p\mid n$, which is exactly the definition of a Giuga number. From (\ref{6753867}) it follows that $n\mid \tilde{\chi}(n)$ iff $n\mid (A+m)$, which completes the proof.
\end{proof}

Now we can state the main result. For Carmichael numbers the function $\chi(n)=\tilde{\chi}(n)/n$ is integer-valued iff $n\mid \tilde{\chi}(n)$. By Lemma \ref{7834567} this is equivalent to \(n\) being an odd Giuga number. Combining this with the results for primes and for composite non-Carmichael numbers, we obtain the following theorem.

\begin{theorem}
The function $\chi(n)$ is integer-valued if and only if $n$ is either a prime number or an odd Giuga number.
\end{theorem}

\section{The case \(\chi(p)=p-1\)}

Note that there exist primes for which
\begin{equation}
\chi(p)=p-1.
\label{765097}
\end{equation}
In the range from 3 to 101 we found 13 such numbers:
\begin{equation}
3, 5, 7, 13, 17, 19, 37, 41, 43, 61, 73, 97, 101.
\label{765099997}
\end{equation}
We shall now characterize these numbers.

\begin{definition}
A Nov\'ak-Carmichael number is a positive integer \(n\) such that
\(
a^n \equiv 1 \pmod n
\)
for every integer \(a\) coprime to \(n\).
\end{definition}

The following theorem gives a simple criterion for such numbers (see \cite{Kalmynin}, Theorem 6).
\begin{theorem}
\label{675430988}
A positive integer \(n\) is a Nov\'ak-Carmichael number if and only if \((q-1)\mid n\) for every prime divisor \(q\) of \(n\).
\end{theorem}

For illustration, we list all Nov\'ak-Carmichael numbers in the range from 2 to 100:
\begin{equation}
2, 4, 6, \mathbf{8}, 12, 16, 18, \mathbf{20}, \mathbf{24}, \mathbf{32}, 36, 40, 42, \mathbf{48}, \mathbf{54}, 60, \mathbf{64}, 72, \mathbf{80}, \mathbf{84}, 96, 100.
\label{76007}
\end{equation}
Here the numbers in bold are those for which \(n+1\) is not prime. For the remaining numbers, \(n+1\) is prime, and they exactly give the list (\ref{765099997}) (when \(n+1\le 101\)).

\begin{proposition}
Primes satisfying condition (\ref{765097}) are precisely those for which \(p-1\) is a Nov\'ak-Carmichael number.
\end{proposition}
\begin{proof}
By (\ref{675498}), condition (\ref{765097}) is equivalent to
\(
\gcd(p-1,\operatorname{num}(B_{p-1}))=1.
\)
By the von Staudt-Clausen theorem, a prime \(q\) divides the denominator of \(B_{p-1}\) iff \((q-1)\mid(p-1)\). Hence \(q\) does not divide the numerator \(\operatorname{num}(B_{p-1})\) iff \((q-1)\mid(p-1)\). Thus \(\gcd(p-1,\operatorname{num}(B_{p-1}))=1\) means that for all primes \(q\mid(p-1)\) we have \((q-1)\mid(p-1)\). By Theorem \ref{675430988} this exactly means that \(p-1\) is a Nov\'ak-Carmichael number.
\end{proof}

According to this proposition, the primes (\ref{765099997}) are the first 13 terms of the sequence A337119 (if we exclude \(2\), which appears at the beginning of that sequence, because we consider only odd \(p\ge 3\)).

\section{Classes of primes \(P_\mu\)}

Define the class of primes
\(
P_{\mu}=\{p \text{ prime} : \chi(p)=\mu\}.
\)
Since, by Lemma \ref{875643}, for each $p$ the value $\chi(p)$ is a uniquely determined even integer, every prime $p$ belongs to exactly one set $P_{\chi(p)}$. In other words, the condition $\chi(p)=\mu$ partitions the set of all primes into disjoint classes $P_\mu$.

The following table lists the first ten primes for those classes $P_\mu$ with $\mu\le 100$ for which the first (smallest) prime satisfies condition (\ref{765097}) (which is equivalent to $m_\mu=\mu+1$ and, in particular, $\mu+1$ is prime).
\begin{table}[h]
\centering
\caption{First ten primes in the classes $P_\mu$ with \(m_\mu=\mu+1\).}
\begin{tabular}{|c|c|c|c|c|c|c|c|c|c|c|}
\hline
$P_2$   & 3  & 11 & 23 & 47 & 59  & 71  & 83  & 107 & 131 & 167 \\
\hline
$P_4$   & 5  & 29 & 53 & 149& 173 & 197 & 269 & 293 & 317 & 389 \\
\hline
$P_6$   & 7  & 31 & 67 & 79 & 103 & 139 & 151 & 223 & 283 & 367 \\
\hline
$P_{12}$& 13 & 229& 277& 349& 373 & 709 & 733 & 853 & 877 & 997 \\
\hline
$P_{16}$& 17 & 113& 593& 977& 1553& 2129& 2417& 2609& 2897& 3089 \\
\hline
$P_{18}$& 19 & 199& 307& 523& 739 & 1063& 1171& 1279& 1423& 1531 \\
\hline
$P_{36}$& 37 & 397& 613& 829& 1117& 1549& 1693& 2557& 3637& 3709 \\
\hline
$P_{40}$& 41 & 281& 521& 761& 1481& 1721& 2441& 3881& 5081& 6521 \\
\hline
$P_{42}$& 43 & 211& 463& 547& 967 & 1051& 1303& 1723& 2731& 3067 \\
\hline
$P_{60}$& 61 & 1021&1381&1741&2221 &3181 &3541 &4021 &4261 &5821 \\
\hline
$P_{72}$& 73 &1657&2089&2953&5113 &5689 &7417 &8713 &9433 &10009 \\
\hline
$P_{96}$& 97 &4129&4513&5857&9697 &10273&11617&12577&15073&15649 \\
\hline
$P_{100}$&101 &701 &1301&1901&3701 &6101 &6701 &7901 &10301&13901 \\
\hline
\end{tabular}
\end{table}

Denote by \(m_\mu\) the smallest prime representative of the class \(P_\mu\). Note that there exist non-empty classes \(P_\mu\) for which \(m_\mu=\mu+1\) does not hold. For example,
\[
m_8=89,\; m_{20}=461,\; m_{22}=1871,\; m_{24}=409,\; m_{32}=353,\; m_{46}=20747,\; m_{48}=1489.
\]
For $\mu=10$ no representative was found in the accessible range of computations, which suggests the conjecture that this class (and possibly some others) might be empty (however, a rigorous proof is lacking).

All primes in a given class \(P_\mu\) share a common arithmetic property. Indeed, from formula (\ref{675498}) it follows that for every \(p\in P_\mu\)
\begin{equation}
\gcd(p-1,\operatorname{num}(B_{p-1})) = \frac{p-1}{\mu},
\label{7657764443}
\end{equation}
where \(\operatorname{num}(B_{p-1})\) denotes the numerator of the Bernoulli number \(B_{p-1}\) in reduced form. In particular, this equality shows that \(\mu\mid(p-1)\). Observe that (\ref{7657764443}) is equivalent to the statement that for each prime divisor \(q\) of \(p-1\) we have
\begin{equation}
v_q(\operatorname{num}(B_{p-1})) = v_q(p-1) - v_q(\mu),
\label{76543}
\end{equation}
where \(v_q(m)\) is the exponent of the prime \(q\) in the prime factorization of \(m\).

Thus, membership of a prime \(p\) in the class \(P_\mu\) is equivalent to the condition that for every prime divisor \(q\) of \(p-1\), the exponent of \(q\) in the numerator of \(B_{p-1}\) is exactly \(v_q(\mu)\) less than that in \(p-1\). This gives an arithmetic characterization of the classes \(P_\mu\).

Formula (\ref{76543}) allows us to obtain explicit congruences on primes $p$ belonging to $P_\mu$ for specific values of $\mu$. We give two examples.

\begin{proposition}
For every \(p\in P_2\) we have
\(
p \equiv 3 \pmod 4.
\)
\end{proposition}
\begin{proof}
The numerator of the Bernoulli number \(B_{p-1}\) is always odd. Since \(v_2(\mu)=1\), from (\ref{76543}) we get
\(
v_2(p-1)=1,
\)
i.e. \(p-1\) is divisible by \(2\) but not by \(4\). Therefore every \(p\in P_2\) satisfies
\(
p \equiv 3 \pmod 4.
\)
\end{proof}
The converse is false: for example, \(p=7\) satisfies \(p\equiv 3\pmod 4\), but \(\chi(7)=6\).

\begin{proposition}
For every \(p\in P_4\) we have
\(
p \equiv 5 \pmod{24}.
\)
\end{proposition}
\begin{proof}
Here \(v_2(\mu)=2\), so
\(
v_2(p-1)=2,
\)
meaning \(4\mid(p-1)\) but \(8\nmid(p-1)\). For odd prime divisors $q$ of $p-1$, condition (\ref{76543}) gives $v_q(\operatorname{num}(B_{p-1})) = v_q(p-1)$. This means that the numerator of $B_{p-1}$ contains such $q$ to the same power as $p-1$. If $q$ were in the denominator of $B_{p-1}$, then the numerator would not contain $q$, contradicting $v_q(p-1)>0$. Therefore $q$ is not in the denominator, and by the von Staudt-Clausen theorem $(q-1)\nmid (p-1)$. In particular, for $q=3$: if $3\mid(p-1)$, then $2\mid(p-1)$ (which is true), but then $3$ would be in the denominator and the numerator would not contain $3$, contradicting $v_3(\operatorname{num}(B_{p-1}))=v_3(p-1)>0$. Hence $3\nmid(p-1)$, i.e. $p\not\equiv 1\pmod 3$. Since $p$ is prime and not equal to 3, we have $p\equiv 2\pmod 3$. Combining $p\equiv 5\pmod 8$ and $p\equiv 2\pmod 3$, by the Chinese remainder theorem we get
\(
p \equiv 5 \pmod{24}.
\)
\end{proof}
Analogously, for other $\mu$ one can derive congruences on $p$. However, a detailed investigation of all possible $\mu$ and the derivation of explicit systems of congruences for each of them is a separate task that requires careful enumeration of the prime divisors of $\mu$ and analysis of the behaviour of numerators of Bernoulli numbers. In the present work we restrict ourselves to the above examples, which illustrate the general method, and do not aim at a complete classification of congruences for all $\mu$.

\section{Conclusion}

The Agoh conjecture states that for an integer \(n\geq 1\) the congruence
\(
nB_{n-1}\equiv -1 \pmod n
\)
holds if and only if \(n\) is prime. The Giuga conjecture states that
\[
\sum_{j=1}^{n-1} j^{n-1} \equiv -1 \pmod n
\]
holds if and only if \(n\) is prime. In the work of Kellner \cite{Kellner1} it was proved that the Giuga and Agoh conjectures are equivalent. It is known that any counterexample to the Agoh-Giuga conjecture (i.e. a composite \(n\) satisfying the above congruences) is a Giuga number as defined in the introduction. Thus, the existence of an odd Giuga number is equivalent to the negation of the Agoh-Giuga conjecture. Therefore the non-existence of odd Giuga numbers is a fundamental open problem closely connected with the arithmetic nature of Bernoulli numbers.

In the context of the present work this means that the statement of the main theorem can be informally reformulated as follows: \(\chi(n)\) is integer-valued exactly for primes and for hypothetical counterexamples to the Agoh-Giuga conjecture. Hence the problem of finding odd Giuga numbers is equivalent to finding composite odd numbers for which \(\chi(n)\) is an integer. From a theoretical point of view, the proposed integrality criterion gives a new, explicit and constructive approach to the Agoh-Giuga conjecture, reducing it to the analysis of the integrality of a concrete function built from Bernoulli numbers.

Let us make a few remarks about the results obtained in the last section. The above considerations do not allow a simple conclusion about the existence or emptiness of the classes. For instance, the class $P_{10}$ is conjecturally empty, but a proof of this fact does not follow from elementary reasoning and would require a separate investigation using deeper properties of numerators of Bernoulli numbers. The question of which $\mu$ give non-empty classes remains open and is of interest for further study.

From our results it does not follow that any class \(P_\mu\) contains infinitely many primes. Moreover, proving infinitude even for a fixed \(\mu\) appears to be a difficult problem. The reason is that membership in the class \(P_\mu\) imposes strong restrictions on the exponents of prime divisors of the numerator of \(B_{p-1}\), and the distribution of such \(p\) cannot be described by simple means. Nevertheless, the empirical data suggest that for every even \(\mu\) for which the class \(P_\mu\) is non-empty, the set \(P_\mu\) is infinite. However, a rigorous proof remains an open problem, requiring progress in understanding the arithmetic properties of numerators of Bernoulli numbers.

A separate interesting point is the possible connection of the classes \(P_\mu\) with quadratic forms. It is known that Bernoulli numbers are closely related to class numbers of quadratic fields, which suggests that the condition \(p\in P_\mu\) might be reformulated in terms of representability of \(p\) by certain binary quadratic forms. This direction requires further investigation and is not pursued in the present work.

\end{document}